\documentclass[12pt,a4paper]{article}
\usepackage{ifthen} 
\usepackage{calc} 

\usepackage[german,english]{babel} 

\usepackage[all]{xy}
\usepackage{amsmath}
\usepackage{amssymb}
\usepackage{latexsym}
\usepackage{enumerate}
\usepackage{theorem}
\usepackage{mathrsfs}
\usepackage{helvet}

\usepackage{lastpage} 
\usepackage{prettyref} 
\usepackage[pagebackref]{hyperref} 

\input xy 

\newboolean{finalversion} 
\setboolean{finalversion}{true}

\theoremheaderfont{\scshape}

\newrefformat{chap}{\hyperref[{#1}]{Chapter~\ref*{#1}}}
\newrefformat{sec}{\hyperref[{#1}]{Section~\ref*{#1}}}

\newtheorem{Notation}{Notation}
\newtheorem{Convention}{Convention}
\newtheorem{Def-Prop}{Definition-Proposition}
\newtheorem{Proof-Expl}{Proof-Explanation}
\newtheorem{theorem}{Theorem}[section]
\newrefformat{thm}{\hyperref[{#1}]{Theorem~\ref*{#1}}}
\newtheorem{definition}[theorem]{Definition}
\newrefformat{def}{\hyperref[{#1}]{Definition~\ref*{#1}}}
\newtheorem{lemma}[theorem]{Lemma}
\newrefformat{lem}{\hyperref[{#1}]{Lemma~\ref*{#1}}}
\newtheorem{proposition}[theorem]{Proposition}
\newtheorem{Question-Conjecture}[theorem]{Question-Conjecture}
\newrefformat{prop}{\hyperref[{#1}]{Proposition~\ref*{#1}}}
\newtheorem{corollary}[theorem]{Corollary}
\newrefformat{cor}{\hyperref[{#1}]{Corollary~\ref*{#1}}}
{\theorembodyfont{\upshape}
\newtheorem{examplecore}[theorem]{Example}}
\newrefformat{ex}{\hyperref[{#1}]{Example~\ref*{#1}}}

\newtheorem{remark}[theorem]{Remark}
\newrefformat{rem}{\hyperref[{#1}]{Remark~\ref*{#1}}}

\newcommand{\Id}{\ensuremath{\operatorname{\text{Id}}}}
\newcommand{\Spec}{\ensuremath{\operatorname{Spec}}}

\newcommand{\Mod}{\ensuremath{\operatorname{Mod}}}

\newcommand{\Hom}{\ensuremath{\operatorname{Hom}}}

\newenvironment{proof}{\noindent\textsc{Proof:}}{\hspace*{\fill}
$\blacksquare$\par\vspace{.1cm}}

\newcommand{\mylabel}[1]{\label{#1}\ifthenelse{\boolean{finalversion}}{
  }{\marginpar{\tiny #1}}}  

\setboolean{finalversion}{true}
\begin{document}
\title{A  note on linear algebraic differential equations}
\maketitle
\author{Stefan G\"unther}
\begin{abstract}
The main result in this note is that a very general linear homogenous partial differential operator with algebraic (polynomial) coefficents has no nonzero  algebraic solutions. This result is in particular true for  systems of ordinary linear homogenous differential operators.
\end{abstract}

\tableofcontents
\section{Notation and Conventions}
\begin{Convention} By $\mathbb N$\, we denote the natural numbers, by $\mathbb N_0$\, the set of nonnegative integers.
\end{Convention}
 In this note,all rings and schemes are assumed to be noetherian and all morphisms $q: X\longrightarrow S$\, are assumed to be of finite type.\\

We use multi index notation: if $x_1,...,x_n$\, is a set of variables, we denote  $\underline{x}^{\underline{m}}:= x_1^{m_1}\cdot x_2^{m_2}\ldots \cdot x_n^{m_n}$\, where  $\underline{m}:=(m_1,m_2,\ldots ,m_n)$\, is a multi-index of lenght $n$. By $|\underline{m}|$\, we denote the number $m_1+...+m_n$\,.
The partial derivatives of a function $f(x_1,...,x_m)$\,  in the variables $x_i$\, we denote by 
$\partial^{\mid m\mid}/\partial \underline{x}^{\underline{m}}(f(x_1,...,x_m))$\,.\\
\begin{Notation}If $k\longrightarrow A$\, is a homomorphism of commutative rings and $M$ is an $A$-module, the $N^{th}$ jet-module of $M$ relative to $k$ is denoted by $\mathcal J^N(M/k)$\,.\\
If $q: X\longrightarrow S$\, is a morphism of finite type, $\mathcal F$\, is a coherent sheaf on $X$, the $N^{th}$\, jet module of $\mathcal F$\, relative to $S$ is denoted by $\mathcal J^N(\mathcal F/S)$\,.
\end{Notation} 
\section{Introduction}
Linear partial differential equations are a huge research area for quite a long time. In this note, we want to prove a  modest theorem on systems of linear partial differential equations with polynomial, or, more generally algebraic coefficients.\\ 
Recall that if $k$ is a field and $\mathbb A^n_k$\, is affine $n$-space over $k$, then a homogenous partial linear differential operator of order $N$, 
$$D: \bigoplus_{i=1}^m k[x_1,...,x_n]e_i\longrightarrow \bigoplus_{i=1}^mk[x_1,...,x_n]e_i$$ is given by an $m\times m$-matrix $(D_{ij})$\, where 
$$D_{ij}:=\sum_{I, |I|\leq N}a_{I, i,j}(\underline{x})\partial^{|I|}/\partial \underline{x}^I, 1\leq i,j\leq N,$$ and $a_{I, i,j}(x_1,...,x_n)$\, are polynomials in the variables $x_i$. \\  By standard differential calculus, this corresponds to a $k[x_1,...,x_n]$-linear map 
$$\widetilde{D}: (k[x_1,...,x_n][d^1x_1,...,d^1x_n]/(d^1x_1,...,d^1x_n)^{N+1})^{\oplus m}\longrightarrow k[x_1,...,x_n]^{\oplus m}$$ under the natural correspondence $\widetilde{D}\mapsto \widetilde{D}\circ (d^N_{\mathbb A^n_k/k})^{\oplus m}$\, where $d^N_{\mathbb A^N_k/k}$\, is the $N$-truncated Taylor series expansion,
$$k[x_1,...,x_n]\longrightarrow k[x_1,...,x_n][d^1x_1,...,d^1x_n]/(d^1x_1,...,d^1x_n)^{N+1},$$ sending $x_i$\, to $x_i+d^1x_i$\,. \\
Here, the differential operator $D_{ij}$\, corresponds to the $k[\underline{x}]$-linear map, that sends  $(\underline{d^1x})^I$\,  to $a_{I,i,j}$\,. This is a standard calculation. \\
The $k$-algebra $k[\underline{x}][\underline{d^1x}]/(\underline{d^1x}^{N+1})$\, is the $N^{th}$ jet module $\mathcal J^N(k[\underline{x}]/k)$\, of $k[\underline{x}]/k$\, and is a $k[\underline{x}]$-algebra which is a free $k[\underline{x}]$-module.\\
 The inverse limit 
\begin{gather*}\mathcal J^{\mathbb N}(k[\underline{x}]/k):=\projlim_{N\in\mathbb N}\mathcal J^N(k[\underline{x}]/k)=k[x_1,...,x_n][[d^1x_1,...,d^1x_n]]\\
\cong k[x_1,...,x_n]\widehat{\otimes}k[x_1,...,x_n]
\end{gather*} is the universal jet algebra of $k[\underline{x}],$\, where the last expression is the tensor product completed with respect to the ideal $I_{k[\underline{x}]/k}$, which is the kernel of the algebra multiplication map $\mu: k[\underline{x}]\otimes_kk[\underline{x}]\longrightarrow k[\underline{x}]$\,and is generated by $( d^1x_i=1\otimes x_i-x_i\otimes 1, i=1,...,n)$\,. 
For $k=\mathbb R,\mathbb C$\, these are just the linear partial differential operators with polynomial coefficients.  The formalism of jet-modules works for every finitely generated $k$-algebra, in particular for localizations $k[\underline{x}]_f$\,, where $f\in k[\underline{x}]$\,.\\
In this note, we want to use the formal algebraic properties  of the jet modules plus elementary base change properties in order to show that the very general system of linear homogenous partial differential operators with polynomial, or more generally rational coefficients $a_{I,J ,i,i}=\frac{p_{I,J, i,j}(\underline{x})}{f^n}$\, possesses no nonzero rational solutions $\frac{s(\underline{x})}{f^n}$\,.
\section{Basic properties of the Jet modules}
Recall that for a homomorphism of commutative rings $k\longrightarrow A$\,, and an $A$-module $M,$\, the $N^{th}$ jet module, $N\in \mathbb N_0$\, is defined as 
$$\mathcal J^N(M/k)=A\otimes_kM/I_{A/k}^{N+1}\cdot (A\otimes_kM),$$ where $I_{A/k}=\ker(\mu): A\otimes_kA\longrightarrow A$\, is the kernel of the multiplication map. The $A$-module structure of $\mathcal J^N(M/k)$\, is given by the first tensor factor. The natural map $d^N_{M/k}: M\longrightarrow \mathcal J^N(M/k)$\,, the universal derivation, is given by the map $m\mapsto \overline{1\otimes m}$\, which is $k$-linear.\\
 It is well known (or by \cite{Guenther}) that the jet-modules are of finite type over $A$ if $A$\, is a $k$-algebra essentially of finite type and satisfy the localization property,
 i.e., for $f\in A$\, we have $\mathcal J^N(M_f/k)=\mathcal J^N(M/k)_f$\,.\\
  Thus, if $f: X\longrightarrow S$\, is a morphism of  noetherian schemes of finite type and $\mathcal F$\, is a coherent $\mathcal O_X$-module, we can define the jet bundle $\mathcal J^N(\mathcal F/S)$\, by patching together the  jet- modules on  an affine covering of $X=\bigcup_{i=1}^n\Spec A_i$\, and $S=\bigcup_{j=1}^m\Spec B_j$, where $f(\Spec A_i)\subset \Spec B_j$\, and this is a  coherent $\mathcal O_X$-module. We make the following standard definition.
 \begin{definition}\mylabel{def:D46} Let $q:X\longrightarrow S$\, be an arbitrary morphism of finite type of noetherian schemes, or more generally of noetherian algebraic spaces and  $\mathcal F_i, i=1,2$\, be  quasicoherent sheaves on $X$. Then, a differential operator of order $\leq N$\, is a $q^{-1}\mathcal O_S$-linear map $D: \mathcal F_1\longrightarrow \mathcal \mathcal F_2$\, that can be factored as
  $\mathcal F_1\stackrel{d^N_{\mathcal F/S}}\longrightarrow \mathcal J^N(\mathcal F_1/S)$\,
  and an $\mathcal O_X$-linear map $\widetilde{D}: \mathcal J^N(\mathcal F_1/S)\longrightarrow \mathcal F_2$\,. \\
  A differential operator of order $N$\, is a differential operator that is of order $\leq N$\, but not of order $\leq N-1$\,.
 \end{definition}
 Thus, in this situation, there is a 1-1 correspondence between differential operators $\mathcal F_1\longrightarrow F_2$\, relative to $S$ and $\mathcal O_X$-linear maps $\mathcal J^N(\mathcal F_1/S)\longrightarrow \mathcal F_2$\,. 
 We need the following easy 
 \begin{lemma}\mylabel{lem:L45} Let $A\longrightarrow B$\, be a homomorphism of rings, $M$ be a $B$-module, and $Q$ be a $\mathcal J^N(B/A)$\,-module and $t: M\longrightarrow Q$\, be a $B$-linear map with respect to the second $B$-module structure on $\mathcal J^N(B/A)$\,. Then, there is a unique homomorphism of $\mathcal J^N(B/A)$\,-modules $\phi:\mathcal J^N(M/A)\longrightarrow Q$\, such that $t=\phi\circ d^N_{M/A}$\,.\\
 More generally, if $q: X\longrightarrow S$\, is a morphism of schemes, $\mathcal F$\, is a coherent $\mathcal O_X$-module and $\mathcal Q$\, a coherent $\mathcal J^N(X/S)$-module and $t: \mathcal F\longrightarrow Q$\, a map that is $\mathcal O_X$-linear with respect to the second $\mathcal O_X$-module structure of $\mathcal Q$\,, there is a unique $\mathcal J^N(X/S)$-linear map $\phi:\mathcal J^N(\mathcal F/k)\longrightarrow Q$ such that $t=\phi\circ d^N_{\mathcal F/S}$\,.
\end{lemma}
\begin{proof} We have $\mathcal J^N(B/A)=B\otimes_AB/\mathcal I_{(B/A)}^{N+1}$\, and natural homomorphisms $p_1,p_2: B\longrightarrow \mathcal J^N(B/A)$\,. Then, by definition, $\mathcal J^N(M/A)=M\otimes_{B,p_2}\mathcal J^N(B/A)$\,. Then, the statement reduces to the easy fact, that, given a homomorphism of rings $k\longrightarrow l, $\,  a $k$-module $M_k$\,  and   an $l$-module $M_l$\, and a $k$-linear homomorphism $M_k\longrightarrow M_l$\,, there is a unique $l$-linear homomorphism $M_k\otimes_{k}l\longrightarrow M_l$\, which follows by the adjunction of restriction and extension of scalars.\\
The arguement in the global case is the same.
\end{proof}

\begin{proposition}\mylabel{prop:P2} (arbitrary push-forwards) Let $X\stackrel{f}\longrightarrow Y\stackrel{\pi}\longrightarrow S$\, be morphisms of schemes and $\mathcal F_i, i=1,2$\, be a quasi coherent sheaves on $X$. Let $D: \mathcal F_1\longrightarrow \mathcal F_2$\, be a differential operator between $\mathcal F_1$ and $\mathcal F_2$\, relative to $S$ of some order $N\in \mathbb N_0$\,. Then $f_*F_1\stackrel{f_*D}\longrightarrow f_*F_2$\, is a differential operator between the quasicoherent sheaves $f_*F_i, i=1,2$\, relative to $S$, where $f_*D$\, is taken in the category of sheaves of $(\pi\circ f)^{-1}\mathcal O_S$-modules on $X$.
\end{proposition}
\begin{proof} Let $D$ be given by 
$$\widetilde{D}\circ d^N_{\mathcal F_1}: \mathcal F_1\longrightarrow \mathcal J^N(\mathcal F_1/S)\longrightarrow \mathcal F_2,$$ where the first map is $(\pi\circ f)^{-1}\mathcal O_S$-linear and $\widetilde{D}$\, is $\mathcal O_X$-linear. Then $f_*d^N_{\mathcal F/S}$\, is an $\pi^{-1}\mathcal O_S$-linear map from $f_*\mathcal F$\, to $f_*\mathcal J^N(\mathcal F/S)$\, and $f_*\widetilde{D}$\, is $f_*\mathcal O_X$\, and thus $\mathcal O_Y$-linear. The morphism $f$ induces a homomorphism 
 $$J^N(f/S): f^*J^N(Y/S)\longrightarrow J^N(X/S)$$ (same as for ordinary K\"ahler differentials) and by adjunction
 $$\mathcal J^N(Y/S)\longrightarrow f_*\mathcal J^N(X/S).$$ 
 The module $f_*\mathcal J^N(\mathcal F_1/S)$\, is an $f_*\mathcal J^N(X/S)$-module and thus an $\mathcal J^N(Y/S)$-module.
 By \prettyref{lem:L45}, there is a unique homomorphism 
 $$\phi: \mathcal J^N(f_*\mathcal F_1/S)\longrightarrow f_*\mathcal J^N(\mathcal F_1/S)$$
  of $\mathcal J^N(Y/S)$-modules such that $f_*d^N_{\mathcal F_1/S}=\phi\circ d^N_{f_*\mathcal F_1/S}$\,. Thus $$f_*D: f_*\mathcal F_1\longrightarrow f_*\mathcal F_2=(f_*\widetilde{D}\circ \phi)\circ d^N_{f_*\mathcal F_1/S}$$ is a linear partial linear differential operator over $S$.
\end{proof} 
\begin{lemma}\mylabel{lem:L2811} (Base change property) Let $A\longrightarrow B$\, and $A\longrightarrow C$\, be homomorpisms  (of finite type) of noetherian rings and $M$\, be a $B$-module. Then, for all $N\in \mathbb N$\, we have 
$$\mathcal J^N(M/A)\otimes_AC\cong \mathcal J^N(M\otimes_BC/C).$$
\end{lemma}
\begin{proof} First, we treat the case, where $M=C$\, and $B=A[\underline{x}]$\, is a free $A$-algebra. Then , it is well known that
\begin{gather*}\mathcal J^N(A[\underline{x}]/A)\otimes_AC=A[\underline{x}][\underline{d^1x}]/(\underline{d^1x})^{N+1}\otimes_AC=\\
C[\underline{x}][\underline{d^1x}]/(\underline{d^1x})^{n+1}=\mathcal J^N(C[\underline{x}]/C)=\mathcal J^N(B\otimes_AC/C).
\end{gather*}
Next, we treat the case of an arbitrary $A$-algebra $B$\,. Let 
$$B=A[x_i| i\in I]/(f_j| j\in J)$$  be a presentation of $B$. Either by common knowledge, or by \cite{Guenther}, we have
$$\mathcal J^N(B/A)=A[\underline{x}][\underline{d^1x}]/((\underline{d^1x_i})^{N+1},(f_j, d^1f_j|j\in J),$$
where $d^1f_j$\, is defined by $f_j+d^1f_j=f_j(\underline{x+d^1x})$\, or by $d^1f_j= \overline{1\otimes f_j-f_j\otimes 1}$\,. Then the result is clear because a presentation of $B\otimes_AC$\, is then given by 
$B\otimes_AC=C[x_i| i\in I]/(f_j| j\in J)$\,.\\
Now, we treat the general case. We fix an isomorphism $\mathcal J^N(B/A)\otimes_AC\cong \mathcal J^N(B\otimes_AC/C)$\,.\\
 There is a derivation 
$$M\otimes_AC\stackrel{d^N_{M/A}\otimes_A\text{Id}_C}\longrightarrow \mathcal J^N(M/A)\otimes_AC$$
which gives by \prettyref{lem:L45} and the above proved isomorphism 
$$\mathcal J^N(B/A)\otimes_AC\cong \mathcal J^N(B\otimes_AC/C)$$
 a $\mathcal J^N(B\otimes_AC/C)$-linear map
$$\phi_M:\mathcal J^N(M\otimes_AC/C)\longrightarrow \mathcal J^N(M/A)\otimes_AC$$\, which is a natural transformation of functors 
$$(A-\Mod)\longrightarrow (\mathcal J^N(B\otimes_AC/C)-\Mod)$$
(that this is a natural transformation of functors follows from the representing property of the jet modules).
We show that $\phi_M$\, is an isomorphism for all $M\in B-\Mod$\,. First, since the construction of the jet-modules commutes with arbitrary direct sums, we get the result for free $B$-modules.\\
Next, both functors are right exact functors from 
$$(A-\Mod)\quad \text{to}\quad \mathcal J^N(B\otimes_AC/C)-\Mod.$$
 Choosing a presentation 
$$B^{\oplus J}\longrightarrow B^{\oplus I}\longrightarrow M\longrightarrow 0,$$
we get the result for general $M\in (B-\Mod)$\, by the five lemma.
\end{proof}
\begin{lemma}\mylabel{lem:L1130} Let $A\longrightarrow B$\, be a homomorphism of commutative rings $M$\, be a $B$-module and $N$\, be an $A$-module. Then, for all $N\in \mathbb N_0$\, there is a canonical isomorphism
$$\phi_N: \mathcal J^N(M\otimes_AN/A)\longrightarrow \mathcal J^N(M/A)\otimes_AN.$$
\end{lemma}
\begin{proof}  We fix the $B$-module $M$\,. The natural map
$$d^N_{M/A}\otimes_A\text{Id}_N: M\otimes_AN\longrightarrow \mathcal J^N(M/A)\otimes_AN$$ is $B$-linear with respect to the second $B$-module structure on $\mathcal J^N(M/A)\otimes_AN$\,. By \prettyref{lem:L45}, there is a unique homomorphism  of $\mathcal J^N(B/A)$-modules 
$$\phi_N: \mathcal J^N(M\otimes_AN/A)\longrightarrow \mathcal J^N(M/A)\otimes_AN$$ which is a natural transformation of  right exact functors from $(A-\Mod)$\, to $\mathcal J^N(B/A)-\Mod$\,. We show that $\phi_N$\, is an isomorphism for all $A$-modules $N$.\\
First, if $N=A^{\oplus I}$\, is a free $A$-module, $\phi_N$\, is an isomorphism since the formation of jet-modules commutes with arbitrary direct sums.\\
If $N$\, is arbitrary, choose a free presentation 
$$A^{\oplus J}\longrightarrow A^{\oplus I}\twoheadrightarrow N\longrightarrow 0.$$
Since $\phi_{A^{\oplus J}}$\, and $\phi_{A^{\oplus I}}$\, are isomorphisms, it follows by the five lemma that $\phi_N$\, is an isomorphism.
\end{proof}
\begin{remark}\mylabel{rem:R30111} If $q: X\longrightarrow S$\, is a morphism of noetherian schemes, $D: \mathcal E_1\longrightarrow \mathcal E_2$\, is a differential operator relative to $S$ and $\mathcal F$\, is a quasi coherent $\mathcal O_S$-module, it follows from the lemma just proved that $D\otimes_{\mathcal O_S}\text{Id}_{\mathcal F}$\, is a differential operator on $X$ relative to $S$.
\end{remark}
\begin{lemma}\mylabel{lem:L28117} Let $q: X\longrightarrow S$\, be a morphism of finite type of noetherian schemes , $\mathcal F, \mathcal G$\, be quasi coherent $\mathcal O_X$-modules and $\mathcal F_1$\, be a coherent $\mathcal O_X$-submodule of $\mathcal F$\,. Let $D: \mathcal F\longrightarrow \mathcal G$\, be a differential operator relative to $S$ of some order $N$\,. Then $\mathcal G_1:=\mathcal O_X\cdot D(\mathcal F_1)\subset \mathcal G$\, is a coherent submodule of $\mathcal G$\,.
\end{lemma}
\begin{proof} The operator $D$\, corresponds to an $\mathcal O_X$-linear map $\widetilde{D}: \mathcal J^N(\mathcal F/S)\longrightarrow \mathcal G$\,. Since the construction of the jet-modules is functorial, there is an induced homomorphism of $\mathcal J^N(X/S)$-modules 
$$\mathcal J^N(\mathcal F_1/S)\longrightarrow \mathcal J^N(\mathcal F/S)$$ with coherent image $\mathcal J_1\subset \mathcal J^N(\mathcal F/S)$\,, since the jet-module $\mathcal J^N(\mathcal F_1/S)$\, is a coherent $\mathcal O_X$-module.\\
 From the representing property (\prettyref{lem:L45}), it follows that $d^N_{\mathcal F/S}(\mathcal F_1)\subset \mathcal J_1$.\\
 (Put $\mathcal I_1=\mathcal J^N(X/S)\cdot d^N_{\mathcal F/S}(\mathcal F_1)$\, and apply this lemma to obtain that there is an equality $\mathcal I_1= \mathcal J_1$\,. But then $D(\mathcal F_1)\subset \widetilde{D}(\mathcal J_1)\subset \mathcal G$\, is coherent.
 \end{proof} 
\begin{corollary}\mylabel{cor:C28115} Let $q: X\longrightarrow S$\, be a morphism of finite type of noetherian schemes and $\mathcal F_1,\mathcal F_2$\, be quasicoherent sheaves on $X$\, that are both countable unions of its coherent subsheaves. Let $D: \mathcal F_1\longrightarrow \mathcal F_2$\, be a differential operator relative to $S$ of some order $N\in \mathbb N$\,. Then, $D$ can be written as a union (inductive limit) of differential operators $D_n: \mathcal F_1^n\longrightarrow \mathcal F_2^n, n\in \mathbb N$\,.
\end{corollary}
\begin{proof} Let $\mathcal F=\bigcup_{n\in \mathbb N}\mathcal F_1^n$\,. By the previous corollary,  for all $n\in \mathbb N$\,, $\mathcal F_2^n:=\mathcal O_X\cdot D(\mathcal F_1^n)$\, is a coherent subsheaf of $\mathcal F_2$\, and we put 
$$D_n: \mathcal F_1^n\longrightarrow \mathcal F_2^n\quad \forall n\in \mathbb N.$$
\end{proof}  
\section{Differential operators in families}
The basic setting of this section is a  proper morphism  of finite type of noetherian  schemes $q:X\longrightarrow S$\, and coherent sheaves $\mathcal E_i, i=1,2$\, on $X$, flat over $S$, equipped with a differential operator $D: \mathcal E_1\longrightarrow \mathcal E_2$\, of some order $N$  relative to $S$. We want to study the sheaf of $q^{-1}(\mathcal O_S)$-modules $\text{ker}(D)$\,.
\begin{proposition}\mylabel{prop:P77} With notation as above, let $S=\Spec A$\, be affine with $A$  being a noetherian ring. The functor 
\begin{gather*}H^0(X, \text{ker}(D\otimes_A \text{id}_{(-)}): (A-\text{Mod})\longrightarrow (A-\text{Mod})\\
               M\mapsto H^0(X, \text{ker}(D\otimes_A\text{id}_M))
               \end{gather*}
is  naturally equivalent  to a functor $\text{Hom}_A(Q,M) $, where the $A$-module $Q$ is uniquely determined by $D$\, and is  finitely generated.
\end{proposition}
\begin{remark}\mylabel{rem:R29110} Observe that by \prettyref{lem:L1130} for each $A$-module $M$, the $A$-linear map $D\otimes_A\text{Id}_M: \mathcal E_1\otimes_AM\longrightarrow \mathcal E_2\otimes_AM$\, is again a differential operator relative to $A$, so we can apply the technique of Grothendieck to study the kernel of the differential operators on the fibres of the morphism $q$\,.
\end{remark}
\begin{proof} Let $0\longrightarrow M'\longrightarrow M\longrightarrow M''\longrightarrow 0$\, be a short exact sequence of $A$-modules. Since $\mathcal E_i, i=1,2$\, are flat over $A$, we get a short exact sequence of triples 
$$0\longrightarrow \overline{E}\otimes_AM'\longrightarrow \overline{E}\otimes_AM\longrightarrow \overline{E}\otimes_AM''\longrightarrow 0,$$
where $\overline{E}$\, is the triple $\mathcal E_1\stackrel{D}\longrightarrow \mathcal E_2$\,.
Then, taking kernel sheaves is a left exact functor and then, taking $H^0(X,-)$\, is also a left exact functor  and so is the composition.\\
The representability of $H^0(X, \text{ker}(D\otimes_A\text{id}_{-}))$\,  by a functor $\Hom_A(Q,-)$\, follows from \cite{Hartshorne}[chapter III,p.286, Remark 12.4.1]. By standard cohomology and base change, (see \cite{Hartshorne}[chapter III,12, Proposition 12.4]), the functor 
$$(A-\Mod)\longrightarrow (A-\Mod)\quad M\mapsto H^0(X, \mathcal E\otimes_AM)$$ is representable by a functor $\Hom_A(P,-)$\, where $P$\, is an $A$-module of finite type. From the inclusion of functors 
$$H^0(X,\text{ker}(D\otimes_A\text{id}_{(-)}))\hookrightarrow H^0(X, \mathcal E\otimes_A\text{Id}_{(-)})$$
 we get by uniqueness of the representing $A$-modules $Q$ and $P$\,
  a homomorphism of $A$-modules $f:P\longrightarrow Q$\,. I claim that $f$ is surjective.   Let $Q'$\, be the image of $P$ by $f$ in $Q$\,. The  injective homomorphism of $A$-modules
 $\text{Hom}_A(Q,M)\hookrightarrow \text{Hom}_A(P,M)$\, factors as 
 $$\text{Hom}_A(Q,M)\longrightarrow \text{Hom}_A(Q',M)\longrightarrow \text{Hom}_A(P,M).$$
 Thus, the first $A$-linear map $\text{Hom}_A(Q,M)\longrightarrow \text{Hom}_A(Q',M)$\, has to be injective for each $A$-module $M$.\\
  Consider the short exact sequence
 $$0\longrightarrow Q'\longrightarrow Q\longrightarrow Q/Q'\longrightarrow 0.$$  If $Q\neq Q'$\,,put $M=Q/Q'$\, and we get a short exact sequence
 $$0\longrightarrow \text{Hom}_A(Q/Q',Q/Q')\longrightarrow \text{Hom}_A(Q,Q/Q')\longrightarrow \text{Hom}_A(Q',Q/Q').$$
 If $Q/Q'\neq 0$, the first $\Hom$-module is always nonzero, since it containes $\text{id}_{Q/Q'}$\,. But then the second map cannot be injective. Thus it follows $Q=Q'$\, and $P\longrightarrow Q$\, is surjective. Since $P$ is finitely generated, so is $Q$.
 \end{proof}
 From the proof just given it follows the following simple
 \begin{lemma}\mylabel{lem:L2711} Let $A$\, be a noetherian ring , $F_1,F_2: (A-\Mod)\longrightarrow (A-Mod)$\, be two left exact functors and $\phi: F_1\longrightarrow F_2$\, be a natural transformation that is for all $M\in (A-\Mod)$\, injective. Let $Q_1,Q_2$\, be the representing $A$-modules, i.e., 
 $$F_1\cong \Hom_A(Q_1,-)\quad\text{and}\,\, F_2 \cong \Hom_A(Q_2,-).$$
 Then, the homomorphism  $Q_2\longrightarrow Q_1$\, induced by $\phi$\, is surjective.
 \end{lemma}
 \begin{proposition}\mylabel{prop:P117} ( Affine base change property of the kernel)\\
   Let $A\longrightarrow B$\, be a homomorphism of noetherian rings and $X_A\longrightarrow \Spec A$\, be a proper scheme of finite type over $A$, $D: \mathcal E_1\longrightarrow \mathcal E_2$\, be a differential operator relative to $A$\, with $\mathcal E_1, \mathcal E_2$\, coherent sheaves, flat over $A$. Let $Q$\, be the $A$-module, representing the functor 
   $$ H^0(X_A, \ker(D\otimes_A\Id_{-})): A-\Mod\longrightarrow A\Mod,\,\, M\mapsto H^0(X_A, \ker(D\otimes_A\Id_M)).$$
    Let $X_B\longrightarrow \Spec B, \mathcal E_i\otimes_AB,i=1,2$\, and $D_B: \mathcal E_{1,B}\longrightarrow \mathcal E_{2,B}$\, be the based changed data.\\
 Then, the $B$-module $Q_B$, representing the functor 
 $$ H^0(X_B, \ker(D_B\otimes_B\Id_{-})): B-\Mod\longrightarrow B-\Mod,\,\, M\mapsto H^0(X_B, \ker(D_B\otimes_B\Id_M))$$
  is isomorphic to $Q\otimes_AB$\,.
 \end{proposition}
 \begin{remark}\mylabel{rem:R291100} Observe, that we do not claim that the push-forward of the kernel commutes with base change. We only claim that the representing $A$-module $Q$ commutes with base change.
 \end{remark}

 \begin{proof} First, the functor in question is representable by some $B$-module $Q_B$\, since the base changed data fullfill all requirements of \prettyref{prop:P77}. Now let $M$ be a $B$-module.  Let $q: X_B\longrightarrow X_A$\, be the affine projection. Then
 \begin{gather*}H^0(X_B, \ker(D_B\otimes_B\Id_M))=H^0(X_B ,\ker (D\otimes_AB\otimes\Id_M))=\\
 H^0(X_A, q_*(\ker(D\otimes_AB\otimes_B\Id_M))= H^0(X_A, \ker(D\otimes_AM)),
 \end{gather*}
 $M$\, viewed as an $A$-module via the map $A\longrightarrow B$\,. We only have to observe that on $X_B$\, there is an exact sequence of sheaves of $B$-modules
 $$0\longrightarrow \ker(D_B\otimes_B\text{Id}_M)\longrightarrow \mathcal E_{1,B}\otimes_BM\longrightarrow \mathcal E_{2,B}\otimes_BM$$ and that $q_*$\, is a left exact functor. Hence, we get
 $$q_*\ker(D_B\otimes_B\text{Id}_M)=\ker(q_*(D_B\otimes_B\text{Id}_M))=\ker(D\otimes_A\text{Id}_M).$$ Thus,
 $$H^0(X_B, \ker(D_B\otimes_B\otimes\Id_M))\cong \Hom_A(Q,M)\cong \Hom_B(Q\otimes_AB,M)$$
 since any homomorphism of an $A$-module $Q$ into a $B$-module $M$ factors uniquely through $Q\otimes_AB$\,.\\
 The assertion follows by the uniqueness of the representing module $Q_B$\,.
 \end{proof}
  \begin{corollary}\mylabel{cor:C30115} Let $q: X\longrightarrow S$\, be a  proper morphism of finite type of noetherian schemes and $D: \mathcal E_1\longrightarrow \mathcal E_2$\, be a differential operator between coherent sheaves that are flat over $S$. Then, there exists a coherent sheaf $\mathcal Q$\, on $S$ such that for all quasicoherent $\mathcal O_S$-modules $\mathcal F$\, we have a canonical functorial isomorphism of sheaves on $S$
 $$\phi(\mathcal F): q_*(\ker(D\otimes_{\mathcal O_S}\text{Id}_{\mathcal F}))\stackrel{\cong}\longrightarrow \Hom_{\mathcal O_S}(\mathcal Q,\mathcal F).$$
 \end{corollary}
 \begin{proof} Let $S=\bigcup_{i=1}^n\Spec A_i$\, and put $\Spec A_{ij}=\Spec A_i\cap \Spec A_j$\,. Letting $A=A_i$\, and $B=A_{ij}$\, , we know by the previous proposition that there are isomorphisms $\phi_{ij}: Q_{ij}\cong Q_i\otimes_{A_i}A_{ij}$, where $Q_i, Q_{ij}$\, are the representing $A_i, A_{ij}$-modules for the morphism $q$\, restricted to $\Spec A_i, \Spec A_{ij}$\,\,, respetively. Furthermore, these isomorphisms are unique (basically by the Yoneda lemma). Thus, the isomorphisms $\phi_{ij}$\, satisfy the cocycle condition and the finitely generated $A_i$-modules $Q_i$\, glue to a coherent $\mathcal O_S$-module $\mathcal Q$\,. Now, that for each coherent $\mathcal O_S$-module $\mathcal F$\,, the  homomorphism $\phi_{\mathcal F}$\, is an isomorphism  is a local question and this is clear by the affine case $S=\Spec A$\,.
 \end{proof}
 We have the following important
 \begin{proposition}\mylabel{prop:P29113} (Upper semicontinuity of the kernel dimension)\\
 Let $q: X\longrightarrow S$\, be a flat, proper morphism of finite type of noetherian schemes, $\mathcal E_1,\mathcal E_2$\, be coherent sheaves on $X$,  flat over $S$ and $D: \mathcal E_1\longrightarrow \mathcal E_2$\, be a differential operator relative to $S$ of some order $N\in \mathbb N$\,. Then, the function
 $$S\mapsto \mathbb N_0,\,\, s\in S\mapsto h^0(X_s, \text{ker}(D_s))$$
 where $D_s=D\otimes_{\mathcal O_S}\kappa(s)$\,
 is upper semicontinuous on $S$.
 \end{proposition}
 \begin{proof}  First, this is a question local on $S$\,, so we may assume that $S=\Spec A$\, for some noetherian ring $A$. Then, the statement  follows by the same arguement as in \cite{Hartshorne}[chapter III, p.288 Theorem 12.8, case i=0], since the representing $A$-module $Q$ is finitely generated. One can also show this directly, since if $s\in S$\, is a scheme point, we have
\begin{gather*} H^0(X_s, \ker(D_s))= \Hom_A(Q,\kappa(s))= \Hom_A(Q, \mathcal O_{S,s}/\mathfrak{m}_{S,s})=\\
\Hom_A(Q, A_{\mathfrak{p}}/\mathfrak{p}A_{\mathfrak{p}}),
\end{gather*}
 where the prime ideal $\mathfrak{p}\in \Spec A$\, corresponds to $s\in S$\,. But each homomorphism $\phi$\, from a finitely generated  $A$-module $Q$ to an $A_{\mathfrak{p}}$-module $N$\, factors through $Q_{\mathfrak{p}}$\,  (since $Q\otimes_AA_{\mathfrak{p}}=Q_{\mathfrak p}$\,) and, then $\mathfrak{p}Q_{\mathfrak{p}}$\,
  has to go to zero. Thus, 
  $$H^0(X_s, \ker(D_s))=\Hom_A(Q,\kappa(s)) =\check{(Q_{\mathfrak{p}}/\mathfrak{p}\cdot Q_{\mathfrak{p}})},$$ where the dual is taken over $\kappa(s)$\, 
  and it is well known that the fibre dimension of a coherent $\mathcal O_S$-module is an upper semicontinuous function.
 \end{proof}    
 From this, we can derive the following local-to global principle.
 \begin{corollary} Let $q: X\longrightarrow S$\, be a morphism  of finite type of noetherian schemes  with $S=\Spec A$\, be  integral and  affine and $\mathcal E_1,\mathcal E_2$\, be coherent sheaves on $X$ flat over $S$ and $D:\mathcal E_1\longrightarrow \mathcal E_2$\, be a differential operator over $S$. If for a Zariski  dense set of points $s\in S$, the kernel of $D_s: \mathcal E_{1,s}\longrightarrow \mathcal E_{2,s}$\, on $X_s$\, is nonempty, then the kernel of $D$\, is nonempty on $X$.\\
 \end{corollary}
 \begin{proof}  The question is local on $S$, so let $X=\Spec A$\,. The finitely generated $A$-module $Q$ with 
 $$\text{Hom}_A(Q,M)\cong H^0(X,\text{ker}(D\otimes_A\Id_M))$$ can in this case not be a torsion module.  Let $Q':=Q/Q^{tors}$\,. Then 
 $$\text{Hom}_A(Q',A)=\text{Hom}_A(Q,A)= H^0(X,\text{ker}(D))$$ is nonempty because the dual of a torsion free module on an integral scheme is always nonzero torsion free.
 \end{proof} 
 \begin{corollary}\mylabel{cor:C1311}(Specialization Property) Let $(R,\mathfrak{m}, k)$\, be a local ring which is an integral domain with quotient  field $K=K(R)$\,. Let $X_R\longrightarrow \Spec R$\, be a proper morphism of finite type, $\mathcal E_1,\mathcal E_2$\, be coherent sheaves on $X_R$\,, flat over $R$  and $D_R: \mathcal E_R\longrightarrow \mathcal E_R$\, be a differential operator relative to $R$. If $D_K: \mathcal E_K\longrightarrow \mathcal E_K$\, possesses a nonzero solution, then so does $D_k:\mathcal E_k\longrightarrow \mathcal E_k$\,.
 \end{corollary}
 \begin{proof} We know, that the functor 
 $$H^0(X_R,\ker(D\otimes \text{id}_M)): R-\Mod\longrightarrow R-\Mod$$ is naturally equivalent to the functor 
 $$\Hom_R(Q,-): R-\Mod \longrightarrow R-\Mod$$ for some finitely generated $R$-module $Q$.
 Putting $M=K(R),$\, we know that $\Hom_R(Q, K(A))$\, is nonempty which implies that $Q$ cannot be a torsion module and has positive rank. Putting $M=R/\mathfrak{m}$\, we get $H^0(X_k,\ker D_k)\cong \Hom_R(Q, R/\mathfrak{m})$\, which is nonzero since $Q/\mathfrak{m}Q$\, is a nonzero $k$-vector space.
 \end{proof}
 \begin{proposition}\mylabel{prop:P13112} Let $A$ be a  noetherian ring and $\mathbb A_A^n\longrightarrow \Spec A$\, be affine relative $n$-space over $\Spec A$\,. Let $D: A[\underline{x}]^{\oplus r}\longrightarrow A[\underline{x}]^{\oplus r}$\, be a differential operator relative to $A$. Then the function
 $$\phi: \Spec A\longrightarrow \mathbb N_0,\,\, s\in \Spec A\mapsto \dim_{\kappa(s)}\ker D_s$$
 is upper semicontinous with respect to the countable Zariski-topology on $\Spec A$\, which is the topology, where closed sets are countable unions of Zariski-closed subsets of $\Spec A$\,.
 \end{proposition}
 \begin{proof} We consider the standard compactification $j: \mathbb A^n_A\hookrightarrow \mathbb P^n_A$\, and we get a differential operator 
 $$j_*D:  j_*A[\underline{x}]^{\oplus r} = \bigcup_{n\in \mathbb N_0}\mathcal O_A(n)^{\oplus r}\longrightarrow j_*A[\underline{x}]^{\oplus r}=\bigcup_{n\in \mathbb N_0}\mathcal O_A(n)^{\oplus r}$$
 (by the push-forward property of differential operators).\\
 For each $n\in \mathbb N$\,, let $D_n=D|_{\mathcal O_A(n)^{\oplus r}}$\,.  By \prettyref{lem:L28117} and \prettyref{cor:C28115} there are minimal integers $m_1^n,...,m_r^n$\, such that $j_*D(\mathcal O_A(n)^{\oplus r})\subseteq \bigoplus_{i=1}^r\mathcal O_A(m^n_i)$\,.
 We know that 
 $$\ker(D)\cong  H^0(\mathbb P^n_A,\ker(j_*(D)))\cong \bigcup_{n\in \mathbb N_0}H^0( \mathbb P^n_A,\ker D_n).$$
 Each $D_n$\, is a differential operator of  coherent sheaves and by \prettyref{prop:P77}, for each $n\in \mathbb N_0$\, there exists a finitely generated $A$-module $Q_n$\, such that 
 $$H^0(\mathbb P^n_A, \ker (D_n\otimes_A\text{Id}_M))\cong \Hom_A(Q_n, M).$$
 For each $n\in \mathbb N_0$\, we have an inclusion of functors 
 $$\ker (D_n\otimes_A\text{Id}_M)\subseteq \ker(D_{n+1}\otimes_A\text{Id}_M).$$ This is so because for each $n\in \mathbb N_0,$\, there are exact sequences
 \begin{gather}0\longrightarrow \mathcal O_A(n)^{\oplus r}\longrightarrow \mathcal O_A(n+1)^{\oplus r}\longrightarrow \mathcal O_{\mathbb P^{n-1}_A}^{\oplus r}(n+1)\longrightarrow 0\,\,\text{and}\\
 0\longrightarrow \bigoplus_{i=1}^n\mathcal O_{\mathbb P^n_A}(m^n_i)\longrightarrow \bigoplus_{i=1}^r\mathcal O_{\mathbb P^n_A}(m^{n+1}_{i})\longrightarrow \bigoplus_{i=1}^r\mathcal O_{H^{n+1}_{i,A}}(m^{n+1}_{i})\longrightarrow 0.
 \end{gather}
 The subscheme $H^{n+1}_{i,A}\subset \mathbb P^n_A$\, is defined over $k$  since  it is defined by the exact sequence
 $$0\longrightarrow \mathcal O_{\mathbb P^n_k}(-m_i^{n+1}+m_i^n)\longrightarrow \mathcal O_{\mathbb P^n_k}\longrightarrow \mathcal O_{H^{n+1}_i}\longrightarrow 0,$$ where the first inclusion is given by multiplication by $x_{n+1}^{m_i^{n+1}-m_i^n}$\,, where $x_{n+1}=0$\, is the hyperplane at infinity. Thus $H^{n+1}_{i}$\, is a thick $\mathbb P^{n-1}_k$\,  and $\mathcal O_{H^{n+1}_{i,A}}=\mathcal O_{H^{N+1}_i}\otimes_kA$\, is thus flat over $A$.\\
 Since the right hand sides of these two exact sequences are $A$-flat, tensoring with the $A$-module $M$ gives exact sequences and 
 $$\ker(D_n\otimes \text{Id}_M)\subseteq \ker(D_{n+1}\otimes_A\Id_M)$$ as claimed. From the inclusion of functors 
 $$\ker(D_n\otimes_A\text{Id}_M)\subseteq \ker(D_{n+1}\otimes_A\text{Id}_M), n\in \mathbb N_0$$ we get by \prettyref{prop:P77} representing  finitely generated $A$-modules $Q_n$\, for the functors 
 $$M\mapsto H^0(\mathbb P^n, \ker (D_n\otimes \text{Id}_M))$$ and by \prettyref{lem:L2711} surjections of finitely generated $A$-modules $Q_{n+1}\twoheadrightarrow Q_n$ and for each $A$-module $M$
 $$\ker(j_*D\otimes_A\Id_M)=\bigcup_{n\in \mathbb N_0}\ker(D_n\otimes_A\Id_M)$$
  as Zariski sheaves on $\mathbb P^n_A$\,.\\
  For each $A$-module $M$, we get
$$H^0(\mathbb P^n_A, \ker j_*D\otimes_AM)= \varinjlim_{n\in \mathbb N_0}\Hom_A(Q_n,M).$$
Putting $M=\kappa(s)=A_{\mathfrak{p_s}}/\mathfrak{p_s}\cdot A_{\mathfrak{p_s}}$\, for $s\in \Spec A$\, , we get that 
$$\ker D_s\cong H^0(\mathbb P^n_{\kappa(s)}, \ker j_*D_s)= \varinjlim_{n\in \mathbb N_0}\Hom_A(Q_n, \kappa(s)).$$
Since $Q_{n+1}\twoheadrightarrow Q_n$\, are surjections, we have that 
$$\dim_{\kappa(s)}\ker(D_s)= \sup_{n\in \mathbb N_0}\dim_{\kappa(s)}\Hom_A(Q_n, \kappa(s))= \sup_{n\in \mathbb N_0}\dim_{\kappa(s)}\ker D_{n,s}.$$
Since  by \prettyref{prop:P29113}, $\ker(D_{n,s})$\, is an upper semicontinuous function on $\Spec A$\, with respect to the Zariski-topology, we get that $\ker D_s$\, is as the countable supremum of upper semicontinuous functions upper semicontinuous with respect to the countable Zariski-topology. \\
(the set $\{s\in \Spec A| \dim\ker (D_s)< r\}$ is the countable intersection of the Zariski-open subsets 
$$\{s\in \Spec A|\dim_{\kappa(s)}(\ker(D_{n,s})) < r\}.)$$
 \end{proof}
 We have the following 
 \begin{corollary}\mylabel{cor:C13118} Let $\mathbb A^n_A$\, be affine $n$-space over $\Spec A$ where $A$ is a noetherian ring and let 
 $$D: A[\underline{x}]^{\oplus r}\longrightarrow A[\underline{x}]^{\oplus r}$$ be a differential operator relative to $A$. If for  each $s$ in a subset $V\subset \Spec A$\, whose closure in the countable Zariski topology is equal to $\Spec A$\, the kernel of $D_s$\, is nonempy, then the kernel of $D$ is a nonzero $A$-module.
 \end{corollary}
 \begin{proof} This follows from the upper-semicontinuity of the dimension of the kernel in the countable Zariski-topology.
 \end{proof}
  We come to our basic theorem, saying roughly that a very general linear partial differential operator on $\mathbb A^n_{\mathbb K}$, or more generally on an affine integral  scheme of finite type over $\mathbb K,$\, $\mathbb K$\, being an uncountable base field, in particular $\mathbb K=\mathbb R,\mathbb C$,  has kernel zero, i.e. no nonzero solutions.\\
  Let $A=\mathbb K[x_1,...,x_m]$\, and consider a  partial linear differential operator $D: A^{\oplus r}\longrightarrow A^{\oplus r}$\, of some order $N$\,. The operator $D$ is then given by an $(r\times r)$-matrix $(D_{ij})$\, with $D_{ij}: A\longrightarrow A$\, of order $\leq N$\,. Thus each $D_{ij}$\, can be written as
  $$D_{ij}:=\sum_{\mid I\mid \leq N} a_{I,i,j}(\underline{x})\cdot \partial^{\mid I\mid}/\partial \underline{x}^I,$$
  where $a_i(\underline{x})$\, is a polynomial in the variables $x_i, i=1,...,m$\,. The most easy way to define a total parameter space of all differential operators  of order $\leq N$\, on $A^{\oplus r}$\, is to  bound the degree of the polynomials $a_{I,i,j}(\underline{x})$\,, say to $M\in \mathbb N$\, and consider the coefficients $a_{J,I, i,j}$ of 
  $$a_{I, i,j}(\underline{x})=\sum_{J}a_{J,I, i,j}\cdot \underline{x}^{J}$$
   as free variables $t_{J,I, i,j}$\,. Let 
  $$B:=\mathbb K[t_{J,I, i,j}\mid \quad |I| \leq N,\,\, |J| \leq M, 1\leq i,j\leq r],$$ 
  and $\Spec (A\otimes_{\mathbb K}B)=\mathbb A^m\times_k\mathbb A^{K}$\, for a certain $K$\,  and we have a universal partial differential operator 
  $$D^u:=(D_{ij}^u): \mathbb K[\underline{x}, \underline{t}]^{\oplus r}\longrightarrow \mathbb K[\underline{x}, \underline{t}]^{\oplus r},$$
   where for $1\leq i,j \leq r$\, we have 
   $$D^u_{ij}=\sum_{I, |I|\leq N} (\sum_{J, \mid J\mid \leq M}t_{I,J, i,j}\cdot \underline{x}^J)\cdot \partial^{|I|}/\partial \underline{x}^I$$
    is a differential operator relative to $B=\mathbb K[\underline{t}],$\, (since the partial derivatives do not involve the variables $t_{I,J, i,j})$\,. The kernel $K:=\ker(D)$\, of $D^u$\, is then a $B=\mathbb K[\underline{t}]$-module.
   Befor we state and prove the theorem, we give an easy lemma which is simple linear algebra.
   \begin{lemma}\mylabel{lem:L2411} With notation as just introduced, let $D:=(D_{ij})$\, be a linear partial differential operator on $A^{\oplus r}$\, of order $\leq N$\, such that $D_{ij}=0 \,\,\text{for} \quad j>i$\, and $D_{i,i}=a_{\emptyset, i,i}[\underline{x}]$\, is a nonzero polynomial. Then $\ker(D)=(\underline{0})$\,.
   \end{lemma}
   \begin{proof} Assume that $(a_1,...,a_r)\in \ker(D)$, where some $a_i$\, is a nonzero polynomial. Let $i$\, be minimal such that $a_i(\underline{x})\neq 0$\,, i.e., 
   $$(a_1,...,a_r)=(0,...,0,a_i,a_{i+1},...,a_r).$$ Then 
 \begin{gather*} D(0,...,0,a_i, a_{i+1}, ...,a_r)=(D_{ij})\cdot (0,,...,0, a_i,...,a_r)=\\
   (D_{1,1}\cdot 0, D_{21}\cdot 0+D_{2,2}\cdot 0,..., D_{i,1}\cdot 0+D_{i,2}\cdot 0+...+D_{i,i}\cdot a_i,....,)
   \end{gather*}
   which is by assumption the zero vector,  in particular $D_{i,i}\cdot a_i=0.$\, But $D_{i,i}$\, was given by a nonzero polynomial and $\mathbb K[\underline{x}]$\, possesses no zero divisors, hence $a_i=0$\, in contradiction to the assumption that $a_i\neq 0$\,. Thus $(a_1,...,a_r)=\underline{0}$\,.
\end{proof}
\begin{theorem}\mylabel{thm:T1185} Let $A=\mathbb K[x_1,...,x_m],\quad\mathbb K=\mathbb R, \mathbb C$\, or an arbitrary uncountable base field. With notation as just introduced, let 
$$D^u=(D_{ij}^u): \mathbb K[\underline{x}, \underline{t}]^{\oplus r}\longrightarrow  \mathbb K[\underline{x}, \underline{t}]^{\oplus r}$$
be the universal differential operator, i.e., the universal family of differential operators of order $\leq N$\, where the polynomial coefficients have degree $\leq M$\,.\\
Let $K$\, be the number of parameter variables $t_{I,J,i,j}$\,.\\
Then for very general $\underline{c}\in \mathbb K^K$\, the kernel $\ker(D_{\underline{c}})=\underline{0}$\, consists only of the zero vector.\\
In particular, if $m=1$\,, i.e., the case of ordinary linear differential algebraic operators, for very general $\underline{c}$\, there is no nonzero algebraic solution.
\end{theorem}
\begin{proof} By \prettyref{prop:P13112}, the function 
$$\ker(D_{\underline{t}}):\mathbb A^K_{\mathbb K}\longrightarrow \mathbb N_0$$
is upper semicontinuous with respect to the countable Zariski-topology on $\mathbb A^K_{\mathbb K}$\,. As shown in the previous lemma there are operators with zero kernel in the above family. The operators with zero kernel are just the operators with dimension of the kernel $<1$,\, which is then  a countable, nonempty  intersection of Zariski-open subsets of $\mathbb A^K_{\mathbb K}$\,.  The complement is then a countable union of  Zariski closed proper subsets of $\mathbb A^K_{\mathbb K}$\, which is nowhere dense in the classical topology.
\end{proof}
As a special case, we consider  differential operators with constant coefficients. Let $A=\mathbb K[x_1,...,x_m].$\, A differential operator $A^{\oplus r}\longrightarrow A^{\oplus r}$\, is given by an $r\times r$-matrix $(D_{ij})$\, where each $D_{ij}=\sum_{I, |I|\leq N}a_{I,i,j}\cdot \partial^{|I|}/\partial \underline{x}^I$\, with $a_{I, i,j}\in \mathbb K$\,. We can define the universal family of differential operators with constant coefficients 
\begin{gather*} D^{u,c,r}: (D^{u,c,r}_{ij}):\mathbb K[\underline{x}, t_{I,i,j}\mid |I|\leq N, 1\leq i,j\leq r]^{\oplus r}\longrightarrow \mathbb K[\underline{x}, \underline{t}]^{\oplus r}\\
D_{ij}:=\sum_{I} t_{I,i,j}\cdot \partial^{|I|}/\partial \underline{x}^I,
\end{gather*}
\begin{proposition}\mylabel{prop:P24119} Let $A=\mathbb K[x_1,...,x_m]$\, and consider the  above defined universal differential  operator with constant coefficients $D^{u,c,r}: A\otimes_{\mathbb K}\mathbb K[\underline{t}]^{\oplus r}\longrightarrow A\otimes_{\mathbb K}\mathbb K[\underline{t}]^{\oplus r}$\,. Then for  very general $(\underline{c})\in \mathbb K^K$\,, the differential operator $D_{c}$\, has zero kernel.
\end{proposition}
\begin{proof} This follows from \prettyref{thm:T24112} and the fact that this  universal family of operators with constant coefficients contains the identity operator which has zero kernel.
\end{proof}
It would be very interesting to find more special algebraic families of differential operators with very general kernel zero. By \prettyref{lem:L2411} and \prettyref{thm:T1185} we have e.g. the following 
\begin{corollary}\mylabel{cor:C29116} Let $D^{u,\Delta,M}:\mathbb K[\underline{x}][\underline{t}]^{\oplus r}\longrightarrow \mathbb K[\underline{x}][\underline{t}]^{\oplus r}$\,, the universal family of all differential  lower triangular operators  $(D_{ij})$\, with polynomial coefficents of degree $\leq M$\, with $D_{ij}=0$\, for $j>i$\, and
$$D_{ii}=a_i(\underline{x})+\sum_{I, |I|\leq N}(\sum_{J, |J|\leq M}t_{I,J,i}\underline{x}^J)\cdot \partial^{|I|}/\partial\underline{x}^I$$
where  $a_i(\underline{x})$\, is a fixed nonzero polynomial and the $t_{i,j,J,I}$\, being  the universal parameters.
Then for very general $\underline{c}\in \mathbb K^K$, the operator $\ker(D_{\underline{c}})=\underline{0}$\,is zero.
\end{corollary}
\begin{proof}
\end{proof}
We generalize the above theorem to the following principle.
\begin{theorem}\mylabel{thm:T24112} Let $A$\, be an   integral $\mathbb K$-algebra of finite type and $M$\, be a torsion free $A$-module of rank $r$\,  and let $D: M\otimes_{\mathbb K}B\longrightarrow M\otimes_{\mathbb K}B$\, be a family of differential operators on $M$\, of order $\leq N$\,, where $B$\, is an integral $\mathbb K$-algebra of finite type. Assume that the family $D$ contains the identity operator, or some other differential operator with kernel zero. Then, for very general $\mathfrak{m_s}\in \Spec B$\,, the operator  $D_s=D\otimes_B\kappa(s)$\, has zero kernel.
\end{theorem}
\begin{proof}  First, let $\Spec A_0\subset \Spec A$\, be a Zariski open subset, where $M$\, is a free $A$-module. It suffices to show, that the family $D$\, restricted to $\Spec A_0$\, has zero kernel for very general $s\in \Spec B$\,. I.e., we get a family 
$$D^0: M\mid_{\Spec A_0}\otimes_{\mathbb K}B\longrightarrow M\mid_{\Spec A_0}\otimes_{\mathbb K}B.$$ There is an inclusion $\text{res}_{A,A_0}: M\otimes_{\mathbb K}B\hookrightarrow M\mid_{\Spec A_0}\otimes_{\mathbb K}B$\,  since $M$\, was assumed to be torsion free and thus $\ker(D)\subset \ker(D^0)$\,. (Remember that the kernel forms a  Zariski sheaf on $\Spec A$\, which is a subsheaf of $\widetilde{M}$\, and therefore has injective restriction maps $r_{U,V}$\,, since $M$\, is torsion free) .\\ So we are reduced to the case that $M$\, is a free $A$-module. \\
We  now have to find an integral compactification. This is easy. Embedding $\Spec A$\, into some $\mathbb A^n_{\mathbb K}$ and then into $\mathbb P^n_{\mathbb K}$\, and taking the closure of $\Spec A$\, in $\mathbb P^n$\, with the reduced scheme structure, we find an integral compactification of finite type $X$\, of $\Spec A$\,. By blowing up $X\backslash \Spec A$\, we can assume that  $X\backslash \Spec A$\, is  the support of an effective Cartier divisor. Then, 
$$X_B:=X\times_{\mathbb K}\Spec B\quad \text{and}\quad D_B:=D\times_{\mathbb K}\Spec B$$ are flat 
 over $\Spec B$\, and $\Spec B\otimes_{\mathbb K}A=\Spec A\times_{\mathbb K}\Spec B$\, sits inside $X_B$\, via the open immersion $j_B$\, with complement $D_B$\,.  We have that 
 $$j_{B*}A^{\oplus r}=\bigcup_{n\in \mathbb N_0}\mathcal O_{X_B}(nD_B)^{\oplus r}$$
 (just  rational functions on $X_B$\, with poles of finite order along $D_B$\,). So the proof of \prettyref{thm:T1185} carries over verbatim. For each $n\in \mathbb N$\,, there is an integer $m_n$\, such that $j_{B*}D(\mathcal O_{X_B}(nD_B)^{\oplus r})\subset \mathcal O_{X_B}(m_nD_B)^{\oplus r}$\, and the subscheme $H_{n.B}\subset X_B$\, defined by 
 $$0\longrightarrow \mathcal O_{X_B}((m_n-m_{n+1})D_B)\longrightarrow \mathcal O_{X_B}\longrightarrow \mathcal O_{H_{n,B}}\longrightarrow 0$$
 is flat over $B$\, since it is defined over $\mathbb K$\, (by a multiple of the section defining the effective Cartier divisor $D$).\\
  Thus, the function $\ker D_s: \Spec B\longrightarrow \mathbb N_0$\, is upper semicontinuous with respect to the countable Zariski topology. Since the set of all $s\in \Spec B$\, where $\ker(D_s)=\underline{0}$\,  is nonempty (the identity operator has kernel zero) and is the countable intersection of Zariski-open subsets, for very general $s\in \Spec B$, the kernel of $D_s$\, is zero.
\end{proof}
\begin{remark}\mylabel{rem:R28113} Observe, that the compactification is defined over $\mathbb K$\, but that the differential operator $j_{B*}D$\, is defined over $B$. The only thing we need is \prettyref{lem:L28117} and \prettyref{cor:C28115}, saying that a differential operator between quasicoherent sheaves, that are unions of coherent subsheaves can be written as a union of differential operators between these coherent subsheaves (follows from the fact that the relative  jet-modules of a morphism of finite type between noetherian schemes are coherent).
\end{remark}
If $A$\, is an arbitrary integral $\mathbb K$-algebra  of finite type and $M$\, is a torsion free $A$-module, we have that $DO^N(M,M)$\, is a finitely generated $A$-module. We can write $DO^N(M,M)=\bigcup_{n\in \mathbb N}V_n$\, where  all $V_n$\,  are finite dimensional $\mathbb K$-vector spaces containing $\text{Id}_M$\,. Viewing each $V_n$\, as an affine space  $\mathbb A^{K_n}_{\mathbb K}$\, over $\mathbb K$\,, there is a differential operator  
$$D^{u,n}: M\otimes_{\mathbb K}\mathbb K[\underline{t}]\longrightarrow M\otimes_{\mathbb K}\mathbb K[\underline{t}]$$ relative to $\mathbb K[\underline{t}]$.
We make the convention that the point $\underline{0}$\, corresponds to the identity operator, i.e., we write $a'_{\emptyset, i,j}(\underline{x},\underline{t})=1+a_{\emptyset,i,j}(\underline{x},\underline{t})$\,.
 By the above theorem, for each $n\in \mathbb N$\,, for very general $v\in V_n$\,, the kernel of $D^{u,n}_v$\, is zero. We can view $DO^N(M,M)$\, as an infinite affine $\mathbb K$-space endowed with the inductive limit topology, that is an affine
  $\text{Ind}$-scheme over $\mathbb K$\,. The topology does not depend on the choosen exaustion $DO^N(M,M)=\bigcup_{n\in \mathbb N}V_n$\,.
  We consider on $DO^N(M,M)$\, the countable $\text{Ind}$-Zariski-topology, which is the inductive limit of the countable Zariski-topologies on $V_n$.  Then, by \prettyref{thm:T1185}, the set of all $v\in DO^N(M,M)$\, such that $\ker(D_v)=(0)$\, is open with respect to the countable $\text{Ind}$-Zariski-topology. We thus have the following
 \begin{theorem}\mylabel{thm:T24116} With notation as above, with respect to the $\text{Ind}$-Zariski-topology, a very general partial linear differential operator in the infinite dimensional affine space $DO^N(M,M)$\, has kernel zero.
 \end{theorem}
 \begin{proof}
 \end{proof} 
We have the following simple application which makes use of the fact that $\mathbb K^K$\, is a Baire topological space.
\begin{proposition}\mylabel{prop:P28114} Let $A$\, be an integral $\mathbb K$-algebra, $M$\, be a finitely  generated $A$-module and 
$$M\otimes_{\mathbb K}\mathbb K[\underline{t}]\longrightarrow M\otimes_{\mathbb K}\mathbb K[\underline{t}]$$
be a differential operator on $M\otimes_{\mathbb  K}\mathbb K[\underline{t}]$ relative  to $\mathbb K[\underline{t}]$\,, i.e., an algebraic family of differential operators $D_{\underline{c}}$\, with $\underline{c}\in \mathbb K^K$\,.\\
Given $\underline{c}\in \mathbb K^K$\, and $\epsilon >0$\,, there is $\underline{c'}\in B(\underline{c}, \epsilon)$\, such that $\ker(D_{\underline{c'}})=\underline{0}$\,.
\end{proposition}
\begin{proof} This follows from the fact  that $\mathbb K^K$\, with the classical topology is a Baire topological space. Then,   an open  $\epsilon$-ball around some $\underline{c}\in \mathbb K^K$\, cannot be the countable union of Zariski-closed subsets which are nowhere dense.
\end{proof}
\begin{remark}\mylabel{rem:R112} This says, that given a linear partial differential operator on $M$, one can achieve by arbitrary small changes of its coefficients, that the kernel is zero.
\end{remark}
\begin{proposition}\mylabel{prop:P291133} With notation as above, let $V:=DO^N(M,M)=\bigcup_{n\in \mathbb N}V_n$\, (we again make the convention that the identity operator corresonds to the zero vector) and give $V$ the lc-inductive limit topology, where each $V_n$\, is equipped with the classical topology. This is the finest lc-topology on $V$. Then, the set $C\subset V$\,, of all $v\in V$\, with $\ker(D_v)=\underline{0}$\, has empty interior.
\end{proposition}
\begin{proof} Suppose to the contrary that $C^{\circ}\neq \emptyset$. Then $C^{\circ}$\, is a nonempty open subset of $V$ and is of the form $a_0+U$\, where $U$\, is a neighbourhood of zero in $V$\, and contains the absolutely convex hull of open neighbourhoods of zero $0\in U_n\subset V_n$\,, 
$$U\supseteq \text{axc}(U_n\mid n\in \mathbb N).$$
 Let $a_0$\, lie in $V_{n_0}$\,. Then $A^{\circ}\cap V_{n_0}\supseteq a_0+U_{n_0}$\,,
 which is a nonempty open subset in $V_{n_0}$\, which is a contradiction to \prettyref{prop:P28114}. Thus, $C^{\circ}=\emptyset$\,.
 \end{proof}
 An application of the same principle works for families of differential operators with finite dimensional kernel. 
 We first prove the following 
 \begin{lemma}\mylabel{lem:L1121} Consider on $k[x]^{\oplus r}$ a lower triangular differential operator $D=(D_{ij})$\, of the following form:
 \[D=\begin{pmatrix} \partial^1/\partial^1 x & 0 & 0 & \cdots & 0\\
                     D_{21}& \partial^1/\partial^1x  & 0 & \cdots & 0\\
                     \hdotsfor{5}\\
                     D_{r1} & D_{r2} & D_{r3} & \cdots & \partial^1/\partial^1x,
                     \end{pmatrix}
                     \]
                     where each $D_{ij}$\, has constant term zero.
                     Then, the kernel of $D$ consists precisely of the vectors of constant functions and consequently  has finite dimension $r$.
                     \end{lemma}
 \begin{proof} Let $(a_1(x), ...,a_r(x))$\, be in the kernel of $D$. First, we have $\partial^1/\partial^1x(a_1(x))=0$\, and thus $a_1(x)=c_1\in \mathbb K$\,. Suppose by way of induction, we have proven that $a_i(x)=c_i $\, for $1\leq i\leq m<r$\,. Then 
 $$\sum_{i=1}^{m+1}D_{m+1,i}a_i(x)= D_{m+1,1}c_1+D_{m+1,2}c_2+...+\partial^1/\partial^1x(a_{m+1}(x))=0$$ from which  it follows $\partial^1/\partial^1x(a_{m+1}(x))=0$\, because all $D_{m+1,i}$\, have zero constant term. Thus $a_{m+1}(x)=c_{m+1}$\, and the claim is proved.
 \end{proof}
 \begin{proposition}\mylabel{prop:P1124} Consider the following algebraic family of ordinary homgenous linear  differential operators $$D=(D_{ij}): \mathbb K[x,\underline{t}]^{\oplus r}\longrightarrow  \mathbb K[x,\underline{t}]^{\oplus r}$$
 with the following properties.
 \begin{enumerate}[1]
 \item $D_{\underline{0}}$\, is lower triangular and $D_{ii,\underline{0}}= \partial^1/\partial^1x$\, for $1\leq i\leq r$\,;
 \item for all $\underline{c}\in \mathbb A^K_{\mathbb K}$\,, $D_{\underline{c}, i,j}$\, has no constant term.
 \end{enumerate}
 Then, the very general operator in this family has finite nonzero kernel of dimension $1\leq \dim(\ker(D_{\underline{c}}))\leq r.$
 \end{proposition}
 \begin{proof} Obviously under the assumtions made, for each operator $D_{\underline{c}}$\,, the kernel  contains  the vectors of constant functions. By the previous lemma, the family contains operators with kernel dimension $r$. The result follows by upper semicontiuity of the kernel dimension in the countable Zariski-topology.
 \end{proof}
 This gives a lot of examples of ordinary linear differential operators with kernel dimension $r$.
 I think now the principle has become clear. Take any linear partial differential operator where we know the kernel and then vary it in an algebraic family to produce new operators with this property.\\
There are furthermore relatively simply methods to produce more families of differential operators with nonzero, or if $n\geq 2$\,, with infinite dimensional kernel. 
The most natural familiy of operators on $\mathbb K[\underline{x}]^{\oplus r}\longrightarrow \mathbb K[\underline{x}]^{\oplus r}$\, with infinite dimensional kernel is given by all $D=(D_{ij})$, where  
$$D_{ij}=\sum_{I, |I|\leq N}a_{I,i,j}\cdot \partial^{|I|}/\partial \underline{x}^I,$$
where all $a_{I,i,j}=0$ if $i_1 \neq  0$,
 where $I=(i_1,...,i_n)$\, and $a_{i,j,\emptyset}=0$\,. Then, $\mathbb K[x_1]^{\oplus r}\in \ker(D)$\, and $\ker(D)$\, is infinite dimensional. It is easy to construct the universal family which is a linear affine subspace $L$ in the total parameter space.\\
The idea is that the presentation of a differential operator $D:=(D_{ij})$\,, depends 
\begin{enumerate}[a]
\item On the choice of algebraic coordinates on $\mathbb A^n_{\mathbb K}$\,, i.e., polynomial functions $x_{i'}:=\phi_i(\underline{x})$\, such that 
$$\phi: =\underline{\phi}: \mathbb K[\underline{x}]\longrightarrow \mathbb K[\underline{x}]\quad x_i\mapsto \phi_i(\underline{x})$$
is an automorphims of the free $\mathbb K$-algebra $\mathbb K[\underline{x}]$.
\item On the choice of a $\mathbb K[\underline{x}]$-basis of $\mathbb K[\underline{x}]^{\oplus r}$.
\end{enumerate}
E.g., if $A=\mathbb K[x]_f, f\in \mathbb K[x]$\, and $D: \partial^1/\partial^1 x: A\longrightarrow A$\,, then the absolute term of $D$ is zero, but this depends on the choice of basis for the free $A$-module $A= A\cdot e$\,. If we choose the basis $e'=\frac{1}{f}e$\,, then the same $D$\, in the basis $e'$\, is given by 
$$D'(h)=\frac{1}{f}D(f\cdot h)=\frac{1}{f}\cdot \partial^1/\partial^1x(fh)= \partial^1/\partial^1x (h)+\frac{\partial^1/\partial^1x(f)}{f}\cdot h$$ and the absolute term is nonzero for general $f\in \mathbb K[x]$\,. In the basis $e_i'$\, this is the same differential operator and has in particular the same kernel dimension. Putting $e_i'=e_i$\,, we get a new different operator with the same kernel.\\
The arguement with the algebraic basis of $\mathbb K[\mathbb A^n_{\mathbb K}]$\, is the same. Rewrite a given differential operator in a new basis $x_i'=\phi_i(\underline{x})$\, (there we need the transformation behavior of the differentials $d^1\phi_i(\underline{x})$\, which is given by the Taylor formulas  and put $x_i'=x_i$\, to get a new differential operator.\\
Let $\text{GL}^r(\mathbb K[\underline{x}])$\,  be the general linear group of all matrices $(p_{ij})$\, such that $\det(p_{ij})$\, is invertibel in $\mathbb K[\underline{x}]$\,. If $\mathbb K =\mathbb C,$\, since any nonconstant polynomial possesses a nonzero solution, the determinant $\det((p_{ij}))$\,
 has to be a nonzero complex number. We do not view this as a group scheme over $\mathbb K[\underline{x}]$\, but as an "infinite dimensional algebraic group" over $\mathbb K$\,.\\
Likewise, by $\text{Gl}^n_2(\mathbb K)$\, we denote the automorphism group of the free $\mathbb K$-algebra $\mathbb K[x_1,...,x_n]$\,. This can also be viewed as in infinite dimensional algebraic group. We then have the following 
\begin{theorem}\mylabel{thm:T312} Let $V=DO^N(\mathbb K[\underline{x}]^{\oplus r})$\, be the universal parameter space of linear partial differential operators $D: \mathbb K[\underline{x}]^{\oplus r}\longrightarrow \mathbb K[\underline{x}]^{\oplus r}$\,. Then, the group $\text{Gl}^r(\mathbb K[\underline{x}])$\, and the group $\text{Gl}^n_2(\mathbb K)$\, act on $V$\,, where the first action is given by 
$$A\circ D= A^{-1}\cdot D\cdot A\quad \text{for}\quad A\in \text{Gl}^r(\mathbb K[\underline{x}])\quad\text{and}\quad D\in V.$$
\end{theorem}
\begin{corollary}\mylabel{cor:C3124} Let $L\subset V=DO^N(\mathbb K[\underline{x}]^{\oplus r})$\, be the familiy of differential operators $D=D_{ij}$\, with $a_{\emptyset,i,j}=0$\, and $a_{I,i,j}=0$\, if $i_1\neq 0$, which has infinite dimensional kernel. 
 For arbitrary $g\in \text{Gl}^r(\mathbb K[\underline{x}])$\, or $g\in \text{Gl}^n_2(\mathbb K)$\,, we get families $g\circ L\subset V$\, with infinite dimensional kernel.
\end{corollary}
Observe, if the automorphism $\underline{\phi}$\, simply permutes the variables, we get affine subspaces $L_j$\, for each $j=1,...,n$\,.\\
We also know that the cardinality of all families with infinite dimensional kernel is countable, so the quotient of $G$ by the  stablizer of $L\subset V$\, must be discrete.\\
A very important question would be the following
\begin{Question-Conjecture} Let $V^{r,n,N}:=DO^N(\mathbb K[\underline{x}]^{\oplus r})$\, be the universal parameter space of linear partial differential operators on $\mathbb K[x_1,...,x_n]^{\oplus r}$\,  of order $\leq N$\, and $L\subset V^{r,N,n}$\, be the affine subspace of all $(D_{ij})$\, with $a_{\emptyset, i,j}=0$\, and $a_{I, i,j}=0$\, if $i_1\neq 0$\,. As $g$ runs through $\text{Gl}^r(\mathbb K[\underline{x}])$\, or $\text{Gl}^n_2(\mathbb K)$\,, or the correctly defined (semi)-direct product, are the families $g\circ L$\, all families with infinite dimensional algebraic kernel?
\end{Question-Conjecture}
As  a final application of \prettyref{prop:P29113} we  treat the complete case.
\begin{proposition}\mylabel{prop:P24110} Let $X$ be a complete algebraic scheme of finite type over a  base  field $k$, or, more generally a complete algebraic space of finite type over $k$ and $\mathcal E$\, be a coherent sheaf on $X$. Let $N\in \mathbb N$\, be given. Then, the general  differential operator $D: \mathcal E\longrightarrow \mathcal E$\,   has zero kernel. 
\end{proposition}
\begin{proof} The sheaf 
$$DO^N(\mathcal E,\mathcal E)=\Hom_{\mathcal O_X}(\mathcal J^N(\mathcal E/k),\mathcal E)$$ is coherent and thus $V:=H^0(X, DO^N(\mathcal E,\mathcal E))$\, is a finite dimensional $k$-vector space.  There is a universal differential operator $D^u: p^*\mathcal E\longrightarrow p^*\mathcal E$\, on $X\times_kV$\, relative to $V$, where 
$$p: X\times_kV\longrightarrow X\quad \text{and}\quad q: X\times_kV\longrightarrow V$$ are the projections. The morphism $q$ is flat and proper, and $p^*\mathcal E$\,  is flat over $V$. Thus, we may apply \prettyref{prop:P29113} (semicontinuity) to conclude that there is a Zariski-open subset $U\subset V$\, such that for $u\in U$\,, $\ker(D_u)=0$\,. The open subset is nonzero since it contains the identity operator $\text{Id}_{\mathcal E}$\,.
\end{proof}
\subsection{Flat extension of differential operators over the complex numbers}
 As any coherent structure over the complex numbers, given a complete complex algebraic variety $X_{\mathbb C},$\, coherent sheaves $\mathcal E_{1,\mathbb C},\mathcal E_{2,\mathbb C}$\, and a differential operator $$ D_{\mathbb C}:\mathcal E_{1,\mathbb C}\longrightarrow \mathcal E_{2,\mathbb C},$$
 by standard technique we can find a finitely generated $\mathbb Z$-algebra $R\subset \mathbb C$\,, an integral variety $X_R$,  proper over $R$ and coherent sheaves $\mathcal E_{1,R},\mathcal E_{2,R}$, flat over $R$  and a relative differential operator $\mathcal D_R: \mathcal E_{1,R}\longrightarrow \mathcal E_{2,R}$\, over $R$, extending $D$ from $\mathbb C$\, to $R$. The only thing one has to observe, is that the jet module $\mathcal J^N(\mathcal E_{1,\mathbb C}/\mathbb C)$\, is a fixed coherent sheaf on $X_{\mathbb C}$\,. We then have the following
 \begin{corollary}\mylabel{thm:T88} With notation as above, let $(X_R, \mathcal E_{1,R},\mathcal E_{2,R}, D_R)$\, be a flat extension of a differential operator $$\mathcal E_{1,\mathbb C},\mathcal E_{2,\mathbb C}, D_{\mathbb C}$$ on $X_{\mathbb C}$\, to the spectrum of a finitely generated  integral $\mathbb Z$-algebra $R$. If for a dense set of points $y\in \Spec R$\, with $\text{char}(\kappa(y))=p>0$\, a  nonzero solution exists for $D_y$\,, then for $D_{\mathbb C}$\, a nonzero solution exists.
 \end{corollary}
 \begin{proof} The representing $R$-module $Q$ cannot be a torsion module, i.e., is nonzero over the generic point of $R$ which corresponds to the quotient field $K(R)\subset \mathbb C$\,. But then $\text{Hom}_R(Q,\mathbb C)=H^0(X_{\mathbb C}, \text{ker}D_{\mathbb C})$\, is nonzero.
 \end{proof}
 This corollary opens up the possibility  to use characteristic $p$-methods in the study of differential operators over the complex numbers, at least in the complete case. The problem with the affine case is that if $R$\, is a finitely generated $\mathbb Z$-algebra, the set of prime ideals is countable, and thus the countable Zariski-topology is the discrete topology. \\
In forthcoming papers, we want to further investigate algebraic families of differential operators on affine integral schemes over $\mathbb K$\, and, in particular study the supports of the kernel sheaves which can also be viewed as projective varieties since the kernel sheaves are invariant under scaling with a constant from the base field. 
\begin{remark}\mylabel{rem:R1125} As the jet-module formalism and cohomology and base change are also available in the complex analytic setting, in principle, the same theorems hold in the complex analytic category. For noncompact base spaces, one obviously faces the problem, that not each manifold or complex analytic space and each differential operator is compactifyable.
\end{remark}

\bibliography{AlgebraicDiffops}
\bibliographystyle{plain}
\noindent
\emph{E-Mail-adress:}verb!stef.guenther2@vodafone.de!
\end{document}